\documentclass[reqno]{amsart}
\usepackage{amsmath}%
\usepackage{amsfonts}%
\usepackage{amssymb}%
\usepackage{graphicx}
\newtheorem{theorem}{Theorem}
\theoremstyle{plain}

\newtheorem{claim}[theorem]{Claim}

\newtheorem{definition}[theorem]{Definition}

\newtheorem{lemma}[theorem]{Lemma}

\newtheorem{proposition}[theorem]{Proposition}

\def\a{\alpha}
\def\eps{\epsilon}
\def\sqe{\sqrt{3\epsilon'}}
\def\sqa{\sqrt{3\alpha'}}
\def\G{\mathcal{G}}

\begin{document}
\title{On multipartite Hajnal-Szemer\'edi theorems}
\thanks{The second author was supported NSA grants H98230-10-1-0165 and H98230-12-1-0283.}
\author{Jie Han}
\address[Jie Han and Yi Zhao]
{Department of Mathematics and Statistics,\newline
\indent Georgia State University, Atlanta, GA 30303}
\email[Jie Han]{jhan22@gsu.edu}
\author{Yi Zhao}
\email[Yi Zhao]{\texttt{yzhao6@gsu.edu}}%
\date{\today}
\subjclass{Primary 05C35, 05C70} %
\keywords{graph packing, Hajnal-Szemer\'edi, absorbing method}%
\begin{abstract}
Let $G$ be a $k$-partite graph with $n$ vertices in parts such that each vertex is
adjacent to at least $\delta^*(G)$ vertices in each of the other parts. Magyar and Martin \cite{MaMa}
proved that for $k=3$, if $\delta^*(G)\ge \frac{2}{3}n +1$ and $n$ is sufficiently
large, then $G$ contains a $K_3$-factor (a spanning subgraph consisting of $n$ vertex-disjoint copies
of $K_3$). Martin and Szemer\'edi \cite{MaSz} proved that $G$ contains a $K_4$-factor when $\delta^*(G)\ge \frac{3}{4}n$ and $n$ is sufficiently large. Both results were proved using the Regularity Lemma.
In this paper we give a proof of these two results by the absorbing method. Our absorbing lemma actually works for all $k\ge 3$ and may be utilized to prove a general and tight multipartite Hajnal-Szemer\'edi theorem.
\end{abstract}

\maketitle

\section{Introduction}

Let $H$ be a graph on $h$  vertices, and let $G$ be a graph on $n$
vertices.  \emph{Packing} (or \emph{tiling}) problems in extremal
graph theory are investigations of conditions under which $G$ must
contain many vertex disjoint copies of $H$ (as subgraphs), where
minimum degree conditions are studied the most. An $H$-matching of $G$
is a subgraph of $G$ which consists of vertex-disjoint copies of
$H$. A {\em perfect $H$-matching}, or {\em $H$-factor}, of $G$ is an
$H$-matching consisting of $\lfloor n/h\rfloor$ copies of $H$.  Let $K_k$ denote the
complete graph on $k$ vertices. The celebrated theorem
of Hajnal and Szemer\'edi \cite{HaSz} says that every $n$-vertex graph $G$  with $\delta(G)\ge ( k-1 )n /k$ contains a $K_k$-factor  (see \cite{KK} for another proof).

Using the Regularity Lemma of Szemer\'edi \cite{Sz}, researchers have generalized this
theorem for packing arbitrary $H$ \cite{AlYu2, KSS-AY, AliYi, KuOs}.
Results and methods for packing problems can be found in the survey of K\"{u}hn and Osthus \cite{KuOs-survey}.

In this paper we consider multipartite packing, which restricts $G$ to be a $k$-partite graph for $k\ge 2$.
A $k$-partite graph is called \emph{balanced} if its partition sets have the same size.
Given a $k$-partite graph $G$, it is natural to consider the minimum partite degree ${\delta}^*(G)$, the minimum degree from a vertex in one partition set to any other partition set. When $k=2$, $\delta^*(G)$ is simply $\delta(G)$. In most of the rest of this paper, the minimum degree condition stands for the minimum partite degree for short.

Let $\G_k(n)$ denote the family of balanced $k$-partite graphs with $n$ vertices in each of its partition sets. It is easy to see (e.g. using the K\"{o}nig-Hall Theorem) that every bipartite graph $G\in \G_2(n)$ with $\delta^*(G)\ge n/2$ contains a $1$-factor. Fischer~\cite{Fischer} conjectured that if $G\in \G_k(n)$ satisfies
\begin{equation}
\label{eq:md1}
{\delta}^*(G)\ge \frac{k-1}{k} n,
\end{equation}
then $G$ contains a $K_k$-factor and proved the existence of an \emph{almost} $K_k$-factor for $k=3, 4$. Magyar and Martin \cite{MaMa} noticed that the condition \eqref{eq:md1} is not sufficient for odd $k$ and instead proved the following theorem for $k=3$. (They actually showed that when $n$ is divisible by $3$, there is only one graph in $\G_3(n)$, denoted by $\Gamma_3(n/3)$, that satisfies \eqref{eq:md1} but fails to contain a $K_3$-factor, and adding any new edge to $\Gamma_3(n/3)$ results in a $K_3$-factor.)

\begin{theorem}[\cite{MaMa}]
There exists an integer $n_0$ such that If $n\ge n_0$ and $G\in \G_3(n)$ satisfies ${\delta}^*(G)\ge 2n/3 +1$, then $G$ contains a $K_3$-factor. \label{thm:MM}
\end{theorem}

On the other hand, Martin and Szemer\'edi \cite{MaSz} proved the original conjecture holds for $k=4$.

\begin{theorem}[\cite{MaSz}]
There exists an integer $n_0$ such that if $n\ge n_0$ and $G\in \G_4 (n)$ satisfies ${\delta}^*(G)\ge 3n/4$,
then $G$ contains a $K_4$-factor. \label{thm:MS}
\end{theorem}

Recently Keevash and Mycroft \cite{KeMy} and independently Lo and Markstr\"{o}m \cite{LoMa} proved that Fischer's conjecture is asymptotically  true, namely, $\delta^*(G)\ge \frac{k-1}{k}n + o(n)$ guarantees a $K_k$-factor for all $k\ge 3$. Very recently, Keevash and Mycroft \cite{KeMy2} improved this to an exact result. 


In this paper we give a new proof of Theorems~\ref{thm:MM} and \ref{thm:MS} by the absorbing method.  Our approach is similar to that of \cite{LoMa} (in contrast, a geometric approach was employed in \cite{KeMy}). However, in order to prove exact results by the absorbing lemma, one needs only assume $\delta^*(G)\ge (1 - 1/k)n$, instead of  $\delta^*(G)\ge (1 - 1/k + \a)n$ for some $\a > 0$ as in \cite{LoMa}. In fact, our absorbing lemma uses an even weaker assumption $\delta^*(G)\ge (1 - 1/k - \a)n$ and has a more complicated absorbing structure.

The absorbing method, initiated by R\"odl, Ruci\'nski, and Szemer\'edi \cite{RRS-di}, has been shown to be effective handling extremal problems in graphs and hypergraphs. One example is the re-proof of Posa's conjecture by Levitt, S\'ark\"ozy, and Szemer\'edi \cite{LSS}, while the original  proof of Koml\'os, S\'ark\"ozy, and Szemer\'edi \cite{KSS-posa} used the Regularity Lemma. Our paper is another example of replacing the regularity method with the absorbing method. Compared with the threshold $n_0$ in Theorems~\ref{thm:MM} and \ref{thm:MS} derived from the Regularity Lemma, the value of our $n_0$ is much smaller.

Before presenting our proof, let us first recall the approach used in \cite{MaMa, MaSz}. Given a $k$-partite graph $G\in \G_k(n)$ with parts $V_1, \dots, V_k$, the authors said that $G$ is $\Delta$-extremal if each $V_i$ contains a subset $A_i$ of size $\lfloor n/k \rfloor$ such that the density $d(A_i, A_j) \le \Delta$ for all $i\ne j$. Using standard but involved graph theoretic arguments, they solved the extremal case for $k=3, 4$ \cite[Theorem 3.1]{MaMa}, \cite[Theorem 2.1]{MaSz}.

\begin{theorem}
\label{thm:ext}
Let $k=3, 4$. There exists $\Delta$ and $n_0$ such that the following holds. Let $n\ge n_0$ and $G\in \G_k(n)$ be a $k$-partite graph satisfying ${\delta}^*(G)\ge (2/3)n+1$ when $k=3$ and \eqref{eq:md1} when $k=4$.  If $G$ is $\Delta$-extremal, then $G$ contains a $K_k$-factor.
 \end{theorem}

To handle the non-extremal case, they proved the following lemma  (\cite[Lemma 2.2]{MaMa} and \cite[Lemma 2.2]{MaSz}).
\begin{lemma}[Almost Covering Lemma]
\label{lem:ac}
Let $k=3, 4$. Given $\Delta>0$, there exists $\alpha>0$ such that for every graph $G\in \G_k(n)$ with
$\delta^*(G)\ge (1 - 1/k)n - \alpha n$ either $G$ contains an almost $K_k$-factor that leaves at most $C = C(k)$ vertices uncovered or $G$ is $\Delta$-extremal.
\end{lemma}

To improve the almost $K_k$-factor obtained from Lemma~\ref{lem:ac}, they used the Regularity Lemma and Blow-up Lemma \cite{Blowup}. Here is where we need our absorbing lemma whose proof is given in Section 2. Our lemma actually gives a more detailed structure than what is needed for the extremal case when $G$ does not satisfy the absorbing property.

We need some definitions. Given positive integers $k$ and $r$, let $\Theta_{k \times r}$ denote the graph with vertices $a_{ij}$, $i=1, \dots, k$, $j=1, \dots, r$, and $a_{ij}$ is adjacent to $a_{i'j'}$ if and only if $i\ne i'$ and $j\ne j'$.  In addition, given a positive integer $t$, the graph $\Theta_{k \times r}(t)$ denotes the blow-up of $\Theta_{k \times r}$, obtained by replacing vertices $a_{ij}$ with sets $A_{ij}$ of size $t$, and edges $a_{ij} a_{i'j'}$ with complete bipartite graphs between $A_{ij}$ and $A_{i'j'}$. Given $\eps, \Delta > 0$ and $t\ge 1$ (not necessarily an integer), we say that a $k$-partite graph $G$ is $(\eps, \Delta)$-approximate to $\Theta_{k\times r}(t)$ if each of its partition sets $V_i$ can be partitioned into $\bigcup_{i=1}^r V_{ij}$ such that $||V_{ij}| - t|\le \eps t$ for all $i$, $j$ and $d(V_{ij}, V_{i'j})\le \Delta$ whenever $i\ne i'$.\footnote{Here we follow the definition of $(\eps, \Delta)$-approximation in \cite{MaMa,MaSz}. It seems natural to require that $d(V_{ij}, V_{i'j'})\ge 1- \Delta$ whenever $i\ne i'$ and $j\ne j'$ as well. However, this follows from $d(V_{ij}, V_{i'j})\le \Delta$ ($i\ne i'$) when $\delta^*(G)\ge (1 - 1/r)rt$.}

\begin{lemma}[Absorbing Lemma]
\label{lem:ab}
Given $k\ge 3$ and $\Delta>0$, there exists $\alpha = \alpha(k, \Delta)>0$ and an integer $n_1>0$ such that the following holds.  Let $n\ge n_1$ and $G\in \G_k(n)$ be a $k$-partite graph on $V_1\cup \dots \cup V_k$ such that $\delta^*(G)\ge (1 - 1/k)n - \alpha n$. Then one of the following cases holds.
\begin{enumerate}
\item $G$ contains a $K_k$-matching $M$ of size $|M|\le 2(k-1)\alpha^{4k-2}n $ in $G$ such that for every $W\subset V\backslash V(M)$ with $|W\cap V_1| = \dots = |W\cap V_k| \le \alpha^{8k-6}n/4$, there exists a $K_k$-matching covering exactly the vertices in $V(M)\cup W$.
\item We may remove some edges from $G$ so that the resulting graph $G'$ satisfies $\delta^*(G')\ge (1 - 1/k)n - \alpha n$ and is $(\Delta/6, \Delta/2)$-approximate to $\Theta_{k\times k}(\tfrac nk)$.
\end{enumerate}

\end{lemma}

The $K_k$-matching $M$ in Lemma~\ref{lem:ab} has the so-called \emph{absorbing}  property: it can absorb \emph{any} balanced set with a much smaller size.

\begin{proof}[Proof of Theorems~\ref{thm:MM} and \ref{thm:MS}]
Let $k= 3, 4$. Let $\alpha \ll \Delta$, where $\Delta$ is given by Theorem~\ref{thm:ext} and $\alpha$ satisfies both Lemmas~\ref{lem:ac} and \ref{lem:ab}. Suppose that $n$ is sufficiently large. Let $G\in \G_k(n)$ be a $k$-partite graph satisfying ${\delta}^*(G)\ge (2/3)n+1$ when $k=3$ and \eqref{eq:md1} when $k=4$.  By Lemma~\ref{lem:ab}, either $G$ contains a subgraph which is $({\Delta}/{6}, {\Delta}/{2})$-approximate to $\Theta_{k\times k}(\tfrac n k)$ or $G$ contains an absorbing $K_k$-matching $M$. In the former case, for $i=1, \dots, k$, we add or remove at most $\frac{\Delta n}{6k}$ vertices from $V_{i1}$ to obtain a set $A_i\subset V_i$ of size $\lfloor n/k \rfloor$. For $i\ne i'$, we have
\begin{align*}
e(A_i, A_{i'}) & \le e(V_{i1}, V_{i'1}) + \frac{\Delta n}{6k} (|A_i| + |A_{i'}|) \\
& \le \frac{\Delta}{2} |V_{i1}| |V_{i'1}| + 2 \frac{\Delta n}{6k} \left\lfloor \frac{n}{k} \right\rfloor \\
& \le \frac{\Delta}{2}  \left(1 + \frac{\Delta}{6} \right)^2 \left(\frac{n}{k} \right)^2 +  \frac{\Delta n}{3 k} \left\lfloor \frac{n}{k} \right\rfloor  \\
& \le \Delta \left\lfloor \frac{n}{k} \right\rfloor \left\lfloor \frac{n}{k} \right\rfloor,
\end{align*}
which implies that $d(A_i, A_{i'})\le \Delta$. Thus $G$ is $\Delta$-extremal. By Theorem~\ref{thm:ext}, $G$ contains a $K_k$-factor.  In the latter case, $G$ contains a $K_k$-matching $M$ is of size $|M|\le 2(k-1)\alpha^{4k-2}n $ such that for every $W\subset V\backslash V(M)$ with $|W\cap V_1| = \dots = |W\cap V_k| \le  \alpha^{8k-6}n/4$, there exists a $K_k$-matching on $V(M)\cup W$.  Let $G' = G\setminus V(M)$ be the induced subgraph of $G$ on $V(G)\setminus V(M)$, and $n' = |V(G')|$. Clearly $G'$ is balanced. As $\alpha \ll 1$, we have
\[
\delta^*(G')\ge \delta^*(G) - |M| \ge \left(1 - \frac{1}{k} \right)n - 2(k-1)\alpha^{4k-2}n \ge \left(1- \frac{1}{k} - \alpha \right) n'.
\]
By Lemma~\ref{lem:ac}, $G'$ contains a $K_k$-matching $M'$ such that $|V(G')\setminus V(M')| \le C$. Let $W= V(G')\setminus V(M')$. Clearly $|W\cap V_1| = \dots = |W\cap V_k|$. Since $C/k\le \a^{8k-6}n/4$ for sufficiently large $n$, by the absorbing property of $M$,  there is a $K_k$-matching $M''$ on $V(M)\cup W$. This gives the desired $K_k$-factor $M'\cup M''$ of $G$.
\end{proof}

\noindent \textbf{Remarks.}

\begin{itemize}
    \item Since our Lemma~\ref{lem:ab} works for all $k\ge 3$, it has the potential of proving a general multipartite Hajnal-Szemer\'edi theorem. To do it, one only needs to prove Theorem~\ref{thm:ext} and Lemma~\ref{lem:ac} for $k\ge 5$.

    \item Since our Lemma~\ref{lem:ab} gives a detailed structure of $G$ when $G$ does not have desired absorbing $K_k$-matching, it has the potential of simplifying the proof of the extremal case. Indeed, if one can refine Lemma~\ref{lem:ac} such that it concludes that $G$ either contains an almost $K_k$-factor or it is approximate to $\Theta_{k\times k}(\tfrac nk)$ and other extremal graphs, then in Theorem~\ref{thm:ext} we may assume that $G$ is actually approximate to these extremal graphs.

    \item Using the Regularity Lemma, researchers have obtained results on packing arbitrary graphs in $k$-partite graphs, see \cite{Zh, HlSc, CzDe, BuZh} for $k=2$ and \cite{MaZh} for $k=3$. With the help of the recent result of Keevash--Mycroft \cite{KeMy} and Lo-Markstr\"om \cite{LoMa}, it seems not very difficult to extend these results to the $k\ge 4$ case (though exact results may be much harder).  However,  it seems difficult to replace the regularity method by the absorbing method for these problems.

\end{itemize}

\section{Proof of the Absorbing Lemma}

In this section we prove the Absorbing Lemma (Lemma~\ref{lem:ab}). We first introduce the concepts of reachability.

\begin{definition}
In a graph $G$, a vertex $x$ is reachable from another vertex $y$ by a set $S\subseteq V(G)$ if both $G[x\cup S]$ and $G[y\cup S]$ contain $K_k$-factors. In this case, we say $S$ connects $x$ and $y$.
\end{definition}

The following lemma plays a key role in constructing absorbing structures. We postpone its proof to the end of the section.

\begin{lemma}[Reachability Lemma]
\label{lem:reach}
Given $k\ge 3$ and $\Delta>0$, there exists $\alpha = \alpha(k, \Delta)>0$ and an integer $n_2>0$ such that the following holds.  Let $n\ge n_2$ and $G\in \G_k(n)$ be a $k$-partite graph on $V_1\cup \dots \cup V_k$ such that $\delta^*(G)\ge (1 - 1/k)n - \alpha n$. Then one of the following cases holds.
\begin{enumerate}
\item For any $x$ and $y$ in $V_i$, $i\in [k]$, $x$ is reachable from $y$ by either at least $\alpha^3 n^{k-1}$ $(k-1)$-sets or at least $\alpha^3 n^{2k-1}$ $(2k-1)$-sets in $G$.
\item We may remove some edges from $G$ so that the resulting graph $G'$ satisfies $\delta^*(G')\ge (1 - 1/k)n - \alpha n$ and is $(\Delta/6, \Delta/2)$-approximate to $\Theta_{k\times k}(\tfrac nk)$.
\end{enumerate}
\end{lemma}

With the aid of Lemma \ref{lem:reach}, the proof of Lemma \ref{lem:ab} becomes standard counting and probabilistic arguments, as shown in \cite{HPS}.
\begin{proof}
[Proof of Lemma \ref{lem:ab}]
We assume that $G$ does not satisfy the second property stated in the lemma.

Given a \emph{crossing} $k$-tuple $T=(v_1,\cdots,v_k)$, with $v_i\in V_i$, for $i=1,\cdots,k$, we call a set $A$ an \emph{absorbing set} for $T$ if both $G[A]$ and $G[A\cup T]$ contain $K_k$-factors. Let $\mathcal{L}(T)$ denote the family of all $2k(k-1)$-sets that absorb $T$ (the reason why our absorbing sets are of size $2k(k-1)$ can be seen from the proof of Claim~\ref{clm:abT} below).


\begin{claim}\label{clm:abT}
For every crossing $k$-tuple T, we have $|\mathcal{L}(T)|>\alpha^{4k-3}n^{2k(k-1)}$.
\end{claim}

\begin{proof}

Fix a crossing $k$-tuple $T$. First we try to find a copy of $K_k$ containing $v_1$ and avoiding $v_2$, $\dots, v_k$. By the minimum degree condition, there are at least
\[
\prod^k_{i=2}\left(n-1-(i-1)\left(\frac1k+\a \right)n\right) \ge \prod^k_{i=2} \left(n - (i-1)\frac n k - ((k - 1)\a n + 1)\right)
\]
such copies of $K_k$. When $n\ge 3k^2$ and $\frac1{\a} \ge 3k^2$, we have $(k-1)\a n + 1 \le n/(3k)$ and thus the number above is at least
\[
\prod^k_{i=2} \left(n - (i-1)\frac n k - \frac{n}{3k}\right) \ge \left(\frac n k\right)^{k-1}, \text{ when }k\ge 3.
\]

Fix such a copy of $K_k$ on $\{v_1,u_2,u_3,\cdots,u_k\}$. Consider $u_2$ and $v_2$.
By Lemma \ref{lem:reach} and the assumption that $G$ does not satisfy the second property of the lemma, we can find at least $\alpha^3 n^{k-1}$ $(k-1)$-sets or $\alpha^3 n^{2k-1}$ $(2k-1)$-sets to connect $u_2$ and $v_2$. If $S$ is a $(k-1)$-set that connects $u_2$ and $v_2$, then $S\cup K$ also connects $u_2$ and $v_2$ for any $k$-set $K$ such that $G[K]\cong K_k$ and $K\cap S = \emptyset$. There are at least
\[
(n-2)\prod^k_{i=2}\left(n-1-(i-1)\left(\frac1k+\a \right)n\right)\ge \frac{n}{2} \left(\frac n k\right)^{k-1}
\]
copies of $K_k$ in $G$ avoiding $u_2$, $v_2$ and $S$. If there are at least $\alpha^3 n^{k-1}$ $(k-1)$-sets that connect $u_2$ and $v_2$, then at least
\[
\alpha^3 n^{k-1} \cdot \frac{n}{2} \left(\frac n k\right)^{k-1} \frac{1}{\binom{2k-1}{k-1}}\ge 2\alpha^4 n^{2k-1}
\]
$(2k-1)$-sets connect $u_2$ and $v_2$ because a $(2k-1)$-set can be counted at most $\binom{2k-1}{k-1}$ times. Since $2\a^4 < \a^3$, we can assume that there are always at least $2\alpha^4 n^{2k-1}$ $(2k-1)$-sets connecting $u_2$ and $v_2$. We inductively choose disjoint $(2k-1)$-sets that connects $v_i$ and $u_i$ for $i=2, \dots, k$. For each $i$, we must avoid $T$, $u_2, \dots, u_k$, and $i-2$ previously selected $(2k-1)$-sets. Hence there are at least $2\alpha^4 n^{2k-1}-(2k-1)(i-1)n^{2k-2}>\alpha^4 n^{2k-1}$ choices of such $(2k-1$)-sets for each $i\ge 2$. Putting all these together, and using the assumption that $\a$ is sufficiently small, we have
$$|\mathcal{L}(T)|\ge \left(\frac n k\right)^{k-1}\cdot (\alpha^4n^{2k-1})^{k-1}>\alpha^{4k-3}n^{2k(k-1)}.$$

\end{proof}

Every set $S\in \mathcal{L}(T)$ is \emph{balanced} because $G[S]$ contains a $K_k$-factor and thus $|S\cap V_1| = \cdots = |S\cap V_k|= 2(k-1)$. Note that there are $\binom{n}{2(k-1)}^k$ balanced $2k(k-1)$-sets in $G$. Let $\mathcal{F}$ be the random family of $2k(k-1)$-sets obtained by selecting each balanced $2k(k-1)$-set from $V(G)$ independently with probability $p:= \alpha^{4k-3}n^{1-2k(k-1)}$.
Then by Chernoff's bound, since $n$ is sufficiently large, with probability $1-o(1)$, the family $\mathcal{F}$ satisfies the following properties:
\begin{align}
&|\mathcal{F}|\le 2\mathbb{E}(|\mathcal{F}|)\le 2p\binom{n}{2(k-1)}^k\le \alpha^{4k-2}n, & \label {eq:1}\\
&|\mathcal{L}(T)\cap\mathcal{F}|\ge \frac{1}2\mathbb{E}(|\mathcal{L}(T)\cap\mathcal{F}|)\ge \frac{1}2p|\mathcal{L}(T)|\ge \frac{\alpha^{8k-6}n}2 \text{ for every crossing }k\text{-tuple }T.&\label {eq:2}
\end{align}

Let $Y$ be the number of intersecting pairs of members of $\mathcal{F}$. Since each fixed balanced $2k(k-1)$-set intersects at most $2k(k-1) \binom{n-1}{2(k-1)-1} \binom{n}{2(k-1)}^{k-1}$ other balanced $2k(k-1)$-sets in $G$,
\[
\mathbb{E}(Y)\le p^2\binom{n}{2(k-1)}^k 2k(k-1) \binom{n-1}{2k-3} \binom{n}{2(k-1)}^{k-1}\le \frac18\alpha^{8k-6}n.
\]
By Markov's bound, with probability at least $\frac{1}{2}$, $Y\le \alpha^{8k-6}n/4$.
Therefore, we can find a family $\mathcal{F}$ satisfying (\ref{eq:1}), (\ref{eq:2}) and having at most $\alpha^{8k-6}n/4$
intersecting pairs. Remove one set from each of the intersecting pairs and the sets that have no $K_k$-factor from $\mathcal{F}$, we get a subfamily $\mathcal{F}'$ consisting of pairwise disjoint absorbing $2k(k-1)$-sets which satisfies $|\mathcal{F}'|\le |\mathcal{F}|\le \alpha^{4k-2}n$ and for all crossing $T$,
\[
|\mathcal{L}(T)\cap\mathcal{F}'|\ge \frac{\alpha^{8k-6}n}2-\frac{\alpha^{8k-6}n}4\ge \frac{\alpha^{8k-6}n}4.
\]

Since $\mathcal{F}'$ consists of disjoint absorbing sets and each absorbing set is covered by a $K_k$-matching, $V(\mathcal{F}')$ is covered by some $K_k$-matching $M$. Since $|\mathcal{F}'|\le \alpha^{4k-2}n$, we have $|M|\le  2k(k-1)\alpha^{4k-2}n/k = 2(k-1)\a^{4k-2}n$.
Now consider a balanced set $W\subseteq V(G)\backslash V(\mathcal{F}')$ such that $|W\cap V_1|=\cdots=|W\cap V_k|\le \alpha^{8k-6}n/4$. Arbitrarily partition $W$ into at most $\alpha^{8k-6}n/4$ crossing $k$-tuples. We absorb each of the $k$-tuples with a different $2k(k-1)$-set from $\mathcal{L}(T)\cap\mathcal{F}'$. As a result, $V(\mathcal{F}')\cup W$ is covered by a $K_k$-matching, as desired.

\end{proof}

The rest of the paper is devoted to proving Lemma \ref{lem:reach}. First we prove a useful lemma. A weaker version of it appears in \cite[Proposition 1.4]{MaSz} with a brief proof sketch. 

\begin{lemma}
\label{lem3}
Let $k\ge 2$ be an integer, $t\ge 1$ and $\epsilon\ll 1$. Let $H$ be a $k$-partite graph on $V_1\cup \dots \cup V_k$ such that $|V_i| \ge (k-1)(1-\epsilon)t$ for all $i$ and each vertex is nonadjacent to at most $(1+\epsilon)t$ vertices in each of the other color classes. Then either $H$ contains at least $\epsilon^2 t^k $ copies of $K_k$, or $H$ is $(16 k^4 \epsilon^{1/{2^{k -2} }},16 k^4 \epsilon^{1/{2^{k-2} }})$-approximate to $\Theta_{k\times (k-1)}(t)$.
\end{lemma}

\begin{proof}
First we derive an upper bound for $|V_i|$, $i\in [k]$.  Suppose for example, that $|V_k|\ge (k-1)(1+\epsilon)t+\epsilon t$. Then if we greedily construct copies of $K_k$ while choosing the last vertex from $V_k$, by the minimum degree condition and $\epsilon\ll 1$, there are at least
\begin{align*}
& |V_1|\cdot (|V_2|-(1+\epsilon)t) \cdots (|V_{k-1}|-(k-2)(1+\epsilon)t)\cdot (|V_k|-(k-1)(1+\epsilon)t)\\
\ge & (k-1)(1-\epsilon)t\cdot (k-2 -k\epsilon)t \cdots (1-(2k-3)\epsilon) t\cdot \epsilon t\\
\ge & (k-1 - \tfrac12) (k-2 - \tfrac12) \cdots (1- \tfrac12) \epsilon t^k \ge \tfrac{\epsilon}2 t^k
\end{align*}
copies of $K_k$ in $H$, so we are done. We thus assume that for all $i$,
\begin{equation}\label{eq:upbd}
|V_i|\le (k-1)(1+\epsilon)t+\epsilon t < (k-1)(1+ 2\epsilon)t.
\end{equation}

Now we proceed by induction on $k$. The base case is $k=2$. If $H$ has at least $\epsilon^2 t^2$ edges, then we are done. Otherwise $e(H)<\epsilon^2 t^2$. Using the lower bound for $|V_i|$, we obtain that
$$
d(V_1,V_2)<\frac{\epsilon^2 t^2}{|V_1|\cdot |V_2|}\le \frac{\epsilon^2}{(1-\epsilon)^2}<\epsilon.
$$
Hence $H$ is $(2\epsilon,\epsilon)$-approximate to $\Theta_{2\times 1}(t)$. When $k=2$,  $16 k^4 \epsilon^{1/{2^{k -2} }}= 256 \eps$, so we are done.

Now assume that $k\ge 3$ and the conclusion holds for $k-1$. Let $H$ be a $k$-partite graph satisfying the assumptions and assume that $H$ contains less than $ \epsilon^2 t^k$ copies of $K_k$.

For simplicity, write $N_i(v)=N(v)\cap V_i$ for any vertex $v$. Let $V_1'\subset V_1$ be the vertices which are in at least $\eps t^{k-1}$ copies of $K_k$ in $H$, and let $\tilde V_1 = V_1\setminus V_1'$. Note that $|V_1'| < \eps t$ otherwise we get at least $\eps^2 t^k$ copies of $K_k$ in $H$. Fix $v_0 \in \tilde V_1$. For $2\le i\le k$, by the minimum degree condition and $k\ge 3$,
\[
|N_i(v_0)| \ge (k-1)(1-\epsilon)t - (1 + \eps) t = (k-2) \left( 1-\frac k{k-2}\eps \right)t \ge (k-2)(1-3\eps)t.
\]
On the other hand, following the same arguments as we used for \eqref{eq:upbd}, we derive that
\begin{equation}\label{eq:upbd2}
|N_i(v_0)| \le (k - 2)(1 + 2\eps t).
\end{equation}

The minimum degree condition implies that a vertex in $N(v_0)$ misses at most $(1 + \eps)t$ vertices in each $N_i(v_0)$. We now apply induction with $k-1$, $t$ and $3\eps$ on $H[N(v_0)]$. Because of the definition of $V_1'$, we conclude that $N(v_0)$ is ($\eps', \eps'$)-approximate to $\Theta_{(k-1)\times (k-2)}(t)$, where
\[
\eps' := 16 (k-1)^4 (3\epsilon)^{1/{2^{k-3}}}.
\]
This means that we can partition $N_i (v_0)$ into $A_{i1} \cup \cdots A_{i(k-2)}$ for $2\le i \le k$ such that
\begin{align}
\label{eq:size22}
& \forall \text{ } 2\le i\le k, \ 1\le j\le k - 2, \quad (1- \eps')t \le |A_{ij}| \le (1+ \eps') t  \qquad \text{and}
\\
\label{eq:den22}
& \forall \text{ } 2\le i< i'\le k, \ 1\le j\le k - 2, \quad d(A_{ij}, A_{i'j})\le \eps'.
\end{align}
Furthermore, let $A_{i(k-1)} := V_i \setminus N(v_0)$ for $i = 2,\cdots, k$. By \eqref{eq:upbd2} and the minimum degree condition, we get that
\begin{equation}\label{eq:size21}
 (1- (3k - 5)\eps)t \le |A_{i(k-1)}| \le (1+ \eps) t,
\end{equation}
for $i = 2,\cdots, k$.

Let $A_{ij}^c = V_i \setminus A_{ij}$ denote the complement of $A_{ij}$. Let $\bar{e}(A, B)= |A| |B| - e(A, B)$ denote the number of non-edges between two disjoint sets $A$ and $B$, and $\bar{d}(A, B)= \bar{e}(A, B)/(|A| |B|)= 1 - d(A, B)$. Given two disjoint sets $A$ and $B$ (with density close to one) and $\alpha>0$, we call a vertex $a\in A$ is $\alpha$-\emph{typical} to $B$ if $\deg_B(a)\ge (1- \alpha)|B|$.

\begin{claim}\label{clm:dia}
Let $2\le i\neq i '\le k$, $1\le j\neq j'\le k-1$.
\begin{enumerate}
\item $d(A_{ij}, A_{i'j'}) \ge 1 - 3\eps'$ and $d(A_{ij}, A_{i'j}^c) \ge 1 - 3\eps'$.
\item All but at most $\sqe$ vertices in $A_{ij}$ are $\sqe$-typical to $A_{i'j'}$; at most $\sqe$ vertices in $A_{ij}$ are $\sqe$-typical to $A^c_{i'j}$.
\end{enumerate}
\end{claim}

\begin{proof}
(1). Since $A^c_{i'j}= \bigcup_{j'\ne j} A_{i'j'}$, the second assertion $d(A_{ij}, A_{i'j}^c) \ge 1 - 3\eps'$ immediately follows from the first assertion $d(A_{ij}, A_{i'j'}) \ge 1 - 3\eps'$. Thus it suffices to show that $d(A_{ij}, A_{i'j'}) \ge 1 - 3\eps'$, or equivalently that $\bar{d}(A_{ij}, A_{i'j'}) \le 3\eps'$.

Assume $j\ge 2$. By \eqref{eq:den22}, we have $e(A_{ij}, A_{i'j})\le \eps' |A_{ij}| |A_{i'j}|$.  So $\bar{e}( A_{ij}, A_{i'j})\ge (1 - \eps')|A_{ij}| |A_{i'j}|$. By the minimum degree condition and \eqref{eq:size22},
\begin{align*}
\bar{e}(A_{ij}, A_{i'j}^c) & \le [(1+\eps)t - (1-\eps')|A_{i'j}|] |A_{ij}| \\
& \le [(1+\eps)t - (1-\eps') (1- \eps')t] |A_{ij}| \\
& < (\eps + 2\eps') t |A_{ij}|,
\end{align*}
which implies that $\bar{e}((A_{ij}, A_{i'j'}) \le  (\eps + 2\eps') t |A_{ij}|$ for any $j' \ne j$ and $1\le j' \le k-1$. By \eqref{eq:size22} and \eqref{eq:size21}, we have $|A_{i'j'}|\ge (1- \eps')t$. Hence
\[
\bar{d}(A_{ij}, A_{i'j'})\le (\eps + 2\eps') \frac{t}{|A_{i'j'}|}\le (\eps + 2\eps') \frac{t}{(1- \eps')t} \le 3\eps',
\]
where the last inequality holds because $\eps \ll \eps' \ll 1$.

(2) Given two disjoint sets $A$ and $B$, if $\bar{d}(A, B)\le \alpha$ for some $\alpha>0$,  then at most $\sqrt{\alpha} |A|$ vertices $a\in A$ satisfy $\deg_B (a) < (1 - \sqrt{\alpha}) |B|$. Hence Part (2) immediately follows from Part (1).
\end{proof}

We need a lower bound for the number of copies of $K_k$ in a dense $k$-partite graph.
\begin{proposition}\label{clm:kk}
Let $G$ be a $k$-partite graph with vertex class $V_1, \cdots, V_k$. Suppose for every two vertex classes, the pairwise density $d(V_i, V_j)\ge 1 - \a$ for some $\a \le (k+1)^{-4}$, then there are at least $\frac12\prod_i |V_i|$ copies of $K_k$ in $G$.
\end{proposition}

\begin{proof}
Given two disjoint sets $V_i$ and $V_j$, if $\bar{d}(V_i, V_j)\le \alpha$ for some $\alpha>0$,  then at most $\sqrt{\alpha} |V_i|$ vertices $v\in V_i$ satisfy $\deg_{V_j} (v) < (1 - \sqrt{\alpha}) |V_j|$. Thus, by choosing typical vertices greedily and the assumption $\a \le (k+1)^{-4}$, there are at least
\[
(1 - \sqrt{\alpha}) |V_1| (1 - 2\sqrt{\alpha}) |V_2| \cdots (1 - k\sqrt{\alpha}) |V_k| > (1 - (1+\cdots + k)\sqrt \a) \prod_i |V_i|> \frac12 \prod_i |V_i|
\]
copies of $K_k$ in $G$.
\end{proof}

\medskip

Let $\eps'' = 2k \sqrt{\eps'}$. Now we want to study the structure of $\tilde{V}_1$.

\begin{claim}\label{clm:v1}
Given $v\in \tilde V_1$ and $2\le i \le k$, there exists $j\in [k-1]$, such that $|N_{A_{ij}}(v)|<\eps''t$.
\end{claim}

\begin{proof}
Suppose instead, that there exist $v\in \tilde{V}_1$ and some $2\le i_0\le k$, such that $|N_{A_{i_0 j}}(v)|\ge \eps''t$ for all $j\in [k-1]$. By the minimum degree condition, for each $2\le i\le k$, there is at most one $j\in [k-1]$ such that $|N_{A_{ij}}(v)|< t/3$. Therefore we can greedily choose $k - 2$ distinct $j_i$ for $i\ne i_0$, such that $|N_{A_{ij_i}}(v)|\ge t/3$. Let $j_{i_0}$ be the the (unique) unused index. Note that
\[
\forall \text{ } i\ne i_0, \quad \frac{|A_{ij_i}|}{|N_{A_{ij_i}}(v)|}\le \frac{(1 + \eps') t}{t/3}<4, \quad \text{ and } \quad \frac {|A_{i_0 j_{i_0}}|} {|N_{A_{i_0 j_{i_0}}}(v)|} \le \frac {(1 + \eps') t}{\eps'' t}<\frac 2{\eps''}
\]
So for any $i\ne i'$, by Claim \ref{clm:dia} and the definition of $\eps''$, we have
\begin{equation}\label{eq:v1}
\bar d(N_{A_{ij_i}}(v), N_{A_{i'j_{i'}}}(v)) \le \frac{3\eps' |A_{ij_i}||A_{i'j_{i'}}|}{|N_{A_{ij_i}}(v)||N_{A_{i'j_{i'}}}(v)|}  \le 3\eps' \cdot 4\cdot \frac{2}{\eps'' } = \frac 6{k^2}\eps''.
\end{equation}
Since $\eps \ll \eps'' \ll 1$, by Proposition~\ref{clm:kk}, there are at least
\[
\frac12 \prod_i N_{A_{ij_i}}(v) \ge \frac12 \cdot \eps''t \left(\frac t3 \right)^{k-2} = \frac{\eps''}{2\cdot 3^{k-2}}t^{k-1} > \eps t^{k-1}
\]
copies of $K_{k-1}$ in $N(v)$, contradicting the assumption $v\in \tilde V_1$.
\end{proof}

Note that if $\deg_{A_{ij}}(v)< \eps'' t$, at least $|A_{i j}| - \eps'' t$ vertices of $A_{i j}$ are not in $N(v)$. By the minimum degree condition, \eqref{eq:size22} and \eqref{eq:size21}, it follows that
\begin{equation}
\label{eq:deg1}
|A^c_{i j} \setminus N(v)| \le (1+ \eps)t - (|A_{ij}| - \eps'' t)\le (1+ \eps)t - (1- \eps') t + \eps'' t \le 2 \epsilon'' t.
\end{equation}
Fix a vertex $v\in \tilde V_1$. Given $2\le i\le k$, let $\ell_i$ denote the (unique) index such that $|N_{A_{i \ell_i}}(v)| < \eps'' t$ (the existence of $\ell_i$ follows from Claim~\ref{clm:v1}).

\begin{claim}\label{clm:li=l2}
We have $\ell_2 = \ell_3 = \cdots = \ell_k$.
\end{claim}

\begin{proof}
Otherwise, say $\ell_2\ne \ell_3$, then we set $j_2 = \ell_3$ and for $3\le i\le k$, greedily choose distinct $j_k, j_{k-1}, \dots, j_3\in [k-1]\setminus \{\ell_3\}$ such that $j_i\neq \ell_i$ (this is possible as $j_3$ is chosen at last). Let us bound the number of copies of $K_{k-1}$ in $\bigcup_{i=2}^{k} N_{A_{i j_i}}(v)$.
By \ref{eq:deg1}, we get $|N_{A_{i j_i}}(v)|\ge |A_{i j_i}| - 2 \eps'' t \ge t/2$ for all $i$.  As in \eqref{eq:v1}, for any $i\neq i'$, we derive that $\bar d(N_{A_{ij_i}}(v), N_{A_{i'j_{i'}}}(v)) \le 3\eps'' \cdot4\cdot 4 = 48\eps''$.
Since $\eps'' \ll 1$, by Proposition~\ref{clm:kk}, we get at least $\frac12 \left( \frac t2  \right)^{k-1} > \eps t^{k-1}$ copies of $K_{k-1}$ in $N(v)$, a contradiction.
\end{proof}

We define $A_{1j} := \{v\in \tilde V_1: |N_{A_{2j}}(v)|< \eps'' t \}$ for $j\in [k-1]$.
By Claims~\ref{clm:v1} and \ref{clm:li=l2}, this yields a partition of $\tilde V_1 = \bigcup_{j=1}^{k-1} A_{1j}$ such that
\begin{equation}\label{eq:den1}
d(A_{1j}, A_{ij})< \frac{\eps''t |A_{1j}|}{|A_{1j}||A_{ij}|} \le\frac{\eps''t}{(1 - \eps')t} < (1+2\eps') \eps'' \quad \text{ for } \ i\ge 2  \ \text{ and } \ j\ge 1.
\end{equation}

By \eqref{eq:size22}, \eqref{eq:size21} and \eqref{eq:deg1}, as $(3k-5)\eps \le \eps'$, we have
\begin{equation}\label{eq:den1dia}
\bar d(A_{1j}, A_{ij'})<\frac {|A_{1j}| 2\eps'' t }{|A_{1j}||A_{ij'}|}\le \frac {2\eps'' t} {(1-\eps')t} < 3\eps'' \quad \text{ for } i\ge 2 \ \text{ and } \ j\neq j'.
\end{equation}
We claim $|A_{1j}| \le (1+ \eps)t + (1+ 2\eps') \eps''|A_{1j}|$ for all $j$. Otherwise, by the minimum degree condition, we have $\deg_{A_{1j}} (v)> (1+2\eps') \eps'' |A_{1j}|$ for all $v\in A_{ij}$, and consequently $d(A_{1j},A_{ij}) > (1+2\eps') \eps''$, contradicting \eqref{eq:den1}. We thus conclude that
\begin{equation}\label {eq:size1u}
|A_{1j}|\le \frac{1+\eps}{1- (1+2\eps') \eps''} t< (1+ 2\epsilon'' )t.
\end{equation}
Since $|V'_1|\le \eps t$, we have $|\bigcup_{j=1}^{k-1} A_{1j}| = |V_1\setminus V'_1| \ge |V_1| - \eps t$. Using \eqref{eq:size1u}, we now obtain a lower bound for $|A_{1j}|$, $j\in[k-1]$:
\begin{equation}\label {eq:size1l}
|A_{1j}| \ge (k-1)(1-\epsilon)t - (k-2)(1+2\epsilon'' )t - \epsilon t \ge (1- 2k\epsilon'') t.
\end{equation}

It remains to show that for $2\le i \ne i'\le k$, $d(A_{i(k-1)},A_{i'(k-1)})$ is small.
Write $N(v_1 \cdots v_m)=\bigcap_{1\le i\le m} N(v_i)$.

\begin{claim} \label{clm:row1}
$d(A_{i(k-1)},A_{i'(k-1)})\le 6 \epsilon''$ for $2\le i,i'\le k$.
\end{claim}

\begin{proof}
Suppose to the contrary, that say $d(A_{(k-1)(k-1)},A_{k(k-1)})> 6 \epsilon''$. 
We first select $k-2$ sets $A_{ij}$ with $1\le i\le k-2$ and $1\le j\le k-2$ such that no two of them are on the same row or column -- there are $(k-2)!$ choices. Fix one of them, say $A_{11},A_{22},\cdots,A_{(k-2)(k-2)}$. We construct copies of $K_{k-2}$ in $A_{11} \cup A_{22} \cup \cdots \cup A_{(k-2)(k-2)}$ as follows. Pick arbitrary $v_1\in A_{11}$.
For $2 \le i\le k-2$, we select $v_i\in N_{A_{ii}}(v_1 \cdots v_{i-1})$ such that $v_i$ is $\sqe$-typical to $A_{(k-1)(k-1)}$, $A_{k(k-1)}$ and all $A_{jj}$, $i< j\le k-2$. By Claim~\ref{clm:dia} and \eqref{eq:deg1}, there are at least $(1- (k-2)\sqe) |A_{ii}| - 2 \eps'' t \ge t/2$ choices for each $v_i$.
After selecting $v_1, \dots, v_{k-2}$, we select adjacent vertices $v_{k-1}\in A_{(k-1)(k-1)}$ and $v_k\in A_{k(k-1)}$ such that $v_{k-1}, v_k \in N(v_1 \cdots v_{k-2})$. For $j\in \{k-1, k\}$, we know that $N(v_1)$ misses at most $2 \eps'' t$ vertices in $A_{j(k-1)}$, and at most $(k-3)\sqe |A_{j(k-1)}|$ vertices of $A_{j(k-1)}$ are not in $N(v_2 \cdots v_{k-2})$. Since $d(A_{(k-1)1},A_{k1})> 6 \epsilon''$ and $\eps''= 2k\sqrt{\eps'}$, there are at least
\begin{align*}
&6 \eps'' |A_{(k-1)(k-1)}| |A_{k(k-1)}| - 2\eps'' t (|A_{(k-1)(k-1)}| + |A_{k(k-1)}|) - 2 (k-3)\sqe |A_{(k-1)(k-1)}| |A_{k(k-1)}|\\
&\ge (6\eps''- 4\eps''- 4 (k-3) \sqrt{\eps'} )|A_{(k-1)(k-1)}| |A_{k(k-1)}|\\
& = 12 \sqrt{\eps'} |A_{(k-1)(k-1)}| |A_{k(k-1)}| \ge 6 \sqrt{\eps'} t^2
\end{align*}
such pairs $v_{k-1}, v_k$. Together with the choices of $v_1, \cdots, v_{k-2}$, we obtain at least $(k-2)!(\tfrac{t}{2})^{k-2} \, 6\sqrt{\eps'} t^2 > \eps t^k$ copies of $K_k$,
a contradiction.
\end{proof}

In summary, by \eqref{eq:size22}, \eqref{eq:size21}, \eqref{eq:size1u} and \eqref{eq:size1l}, we have $(1- 2k\eps'') t \le |A_{ij}| \le (1+ 2\eps'') t$ for all $i$ and $j$. In order to make $\bigcup_{j=1}^{k-1} A_{1j}$ a partition of $V_1$, we move the vertices of $V'_1$ to $A_{11}$. Since $|V_1'|<\eps t$, we still have $| |A_{ij}| - t | \le 2k\eps'' t$ after moving these vertices. On the other hand, by \eqref{eq:den22}, \eqref{eq:den1}, and Claim~\ref{clm:row1}, we have $d(A_{ij}, A_{i'j}) \le 6\eps'' \le 2k\eps''$ for $i\ne i'$ and all $j$ (we now have $d(A_{11}, A_{i1}) \le 2\eps''$ for all $i\ge 2$ because $|A_{11}|$ becomes slightly larger). Therefore $H$ is ($2k\eps'', 2k\eps''$)-approximate to $\Theta_{k\times (k-1)} (t)$. By the definitions of $\eps''$ and $\eps'$,
\[
2k \eps'' = 4k^2 \sqrt{\eps'} = 4k^2 \sqrt{16 (k-1)^4 (3\eps)^{1/2^{k-3}} } \le16 k^4 \eps^{1/ 2^{k-2}},
\]
where the last inequality is equivalent to $(\frac{k-1}{k})^2 \, 3^{1/2^{k-2}} \le 1$ or $3^{1/ 2^{k-1}} \le \frac{k}{k-1}$, which holds because $3\le 1 + \frac{2^{k-1}}{k-1} \le (1 + \frac{1}{k-1})^{2^{k-1}}$ for $k\ge 2$.

This completes the proof of Lemma~\ref{lem3}.
\end{proof}

\bigskip

We are ready to prove Lemma \ref{lem:reach}.

\begin{proof}
[Proof of Lemma \ref{lem:reach}]
First assume that $G\in \G_3(n)$ is minimal, namely, $G$ satisfies the minimum partite degree condition but removing any edge of $G$ will destroy this condition. Note that this assumption is only needed by Claim \ref {clm:row1k2}.

Given $0< \Delta \le 1$, let
\begin{equation}
\label{eq:ad}
\a = \frac1{2k} \left( \frac{\Delta}{24 k(k-1) \sqrt{2k} } \right)^{2^{k-1}}.
\end{equation}
Without loss of generality, assume that $x, y\in V_1$ and $y$ is \emph{not} reachable by $\alpha^3 n^{k-1}$ $(k-1)$-sets or $\alpha^3 n^{2k-1}$ $(2k-1)$-sets from $x$.

For $2\le i\le k$, define
\begin{align*}
   A_{i1}=V_i\cap (N(x)\setminus N(y)), &   \quad A_{ik}=V_i\cap (N(y)\setminus N(x)),  \\
   B_i=V_i\cap (N(x)\cap N(y)), &    \quad A_{i0}=V_i\setminus (N(x)\cup N(y)).
\end{align*}

Let $B=\bigcup_{i\ge 2} B_i$. If there are at least $\alpha^3 n^{k-1}$ copies of $K_{k-1}$ in $B$, then $x$ is reachable from $y$ by at least $\alpha^3 n^{k-1}$ $(k-1)$-sets. We thus assume there are less than $\alpha^3 n^{k-1}$ copies of $K_{k-1}$ in $B$.

Clearly, for $i\ge 2$, $A_{i1}$, $A_{ik}$, $B_i$ and $A_{i0}$ are pairwise disjoint. The following claim bounds the sizes of $A_{ik}$, $B_i$ and $A_{i0}$.

\begin{claim}
\label{clm:Ai1}
\begin{enumerate}
\item $(1- k^2 \alpha) \frac{n}{k} <|A_{i1}|,|A_{ik}|\le (1 + k\alpha)\frac{n}{k}$,
\item $(k-2 - 2k\alpha)\frac{n}{k} \le|B_i|< \left(k -2 + k(k-1)\alpha \right) \frac{n}{k}$,
\item $|A_{i0}|<(k+1)\alpha n$.
\end{enumerate}
\end{claim}

\begin{proof}
For $v\in V$, and $i\in[k]$, write $N_i(v):=N(v)\cap V_i$.
By the minimum degree condition, we have $|A_{i1}|,|A_{ik}|\le (1/k+\alpha)n$.
Since $N_i(x)=A_{i1}\cup B_i$, it follows that
\begin{equation}
\label{eq:Bi}
|B_i|\ge (\tfrac{k-1}k-\alpha)n-(\tfrac1k+\alpha)n. 
\end{equation}

If some $B_i$, say $B_k$, has at least $(\tfrac{k-2}k+(k-1)\alpha)n$ vertices, then there are at least
$ \prod_{i=2}^k |B_i| - (i-2) \left(\frac{1}{k} + \alpha \right) n$ copies of $K_{k-1}$ in $B$.
By \eqref{eq:Bi} and $|B_k|\ge (\tfrac{k-2}k+(k-1)\alpha)n$, this is at least
\begin{align*}
 & \a n \cdot \prod_{i=2}^{k-1} \left( \frac{k-1}{k} - \a \right) n - (i-1) \left(\frac{1}{k} + \alpha \right) n \\
= & \a n \cdot \prod_{i=2}^{k-1} \left( \frac{k-i}{k} -  i\a \right) n \\
\ge & \a n \cdot \prod_{i=2}^{k-1} \left( \frac{k-i - \tfrac12}{k} \right) n \quad \text{because } 2k^2 \a \le 1, \\
\ge &  \a n \cdot \frac12 \left(\frac{n}{k} \right)^{k-2} \\
\ge & \a^2 n^{k-1} \quad \text{because } 2k^{k-2} \a \le 1.
\end{align*}
This is a contradiction.

We may thus assume that $|B_i|<(\tfrac{k-2}k+(k-1)\alpha)n$ for $2\le i \le k$, as required for Part (2). As $N_i(x)=A_{i1}\cup B_i$, it follows that
$$
|A_{i1}|> (\tfrac{k-1}k-\alpha)n-(\tfrac{k-2}k+(k-1)\alpha)n=(\tfrac1k-k\alpha)n.
$$
The same holds for $|A_{ik}|$ thus Part (1) follows. Finally
\[
|A_{i0}| = |V_i| - |N_i(x)| - |A_{ik}| < n-(\tfrac{k-1}k-\alpha)n-(\tfrac1k-k\alpha)n = (k+1)\alpha n,
\]
as required for Part (3).
\end{proof}

Let $t= n/k$ and $\epsilon=2k\alpha$. By the minimum degree condition, every vertex $u\in B$ is nonadjacent to at most $(1+ k\alpha) n/k <(1+\epsilon)t$ vertices in other color classes of $B$. By Claim~\ref{clm:Ai1}, $|B_i|\ge (k-2 - 2k\alpha)\frac{n}{k} =(k-2-\epsilon)t\ge (k-2)(1-\epsilon)t$.  Thus $G[B]$ is a $(k-1)$-partite graph that
satisfies the assumptions of Lemma~\ref{lem3}. We assumed that $B$ contains less than $\alpha^3 n^{k-1}<\eps^2 t^{k-1}$ copies of $K_{k-1}$, so by Lemma \ref{lem3}, $B$ is $(\a', \a')$-approximate to $\Theta_{(k-1)\times(k-2)}(\tfrac nk)$, where
\[
\alpha':=16(k-1)^4 (2k\alpha)^{1/2^{k-3}}.
\]
This means that we can partition $B_i$, $2\le i \le k$, into $A_{i2} \cup \cdots A_{i(k-1)}$ such that $(1- \alpha')\tfrac nk \le |A_{ij}| \le (1+ \alpha') \tfrac nk$ for $2\le j\le k-1$ and
\begin{equation}\label{eq:2den middle}
    \forall \text{ } 2\le i < i'\le k, 2\le j\le k-1, \quad d(A_{ij}, A_{i'j})\le \alpha'.
\end{equation}

Together with Claim~\ref{clm:Ai1} Part (1),  we obtain that (using $k^2 \a\le \a'$)
\begin{equation}
\label{eq:2size}
\forall \text{ } 2\le i \le k, 1\le j\le k, \quad (1- \alpha')\tfrac nk \le |A_{ij}| \le (1+ \alpha') \tfrac{n}{k}.
\end{equation}

Let $A_{ij}^c = V_i \setminus A_{ij}$ denote the complement of $A_{ij}$. The following claim is an analog of Claim~\ref{clm:dia}, and its proof is almost the same -- after we replace $(1+\eps)t$ with $(1+ k\a)n/k$ and $\eps'$ with $\a'$ (and we use $\a \ll \a'$). We thus omit the proof.

\begin{claim}\label{clm:2dia}
Let $2\le i\neq i '\le k$, $1\le j\neq j'\le k$, and $\{j, j'\} \ne \{1, k\}$.
\begin{enumerate}
\item $d(A_{ij}, A_{i'j'}) \ge 1 - 3\alpha'$ and $d(A_{ij}, A_{i'j}^c) \ge 1 - 3\alpha'$.
\item All but at most $\sqa$ vertices in $A_{ij}$ are $\sqa$-typical to $A_{i'j'}$; at most $\sqa$ vertices in $A_{ij}$ are $\sqa$-typical to $A^c_{i'j}$. \qed
\end{enumerate}
\end{claim}

Now let us study the structure of $V_1$. Let $\alpha''=2k \sqrt{\alpha'}$. Recall that $N(xv)=N(x)\cap N(v)$. Let $V_1'$ be the set of the vertices $v\in V_1$ such that there are at least $\a n^{k-1}$ copies of $K_{k-1}$ in each of $N(xv)$ and $N(yx)$. We claim that $|V_1'| < 2\a n$. Otherwise, since a ($k-1$)-set intersects at most $(k-1)n^{k-2}$ other ($k-1$)-sets, there are at least
\[
2\a n\cdot \a n^{k-1} (\a n^{k-1} - (k-1)n^{k-2})  > \a^3 n^{2k-1}
\]
copies of ($2k-1$)-sets connecting $x$ and $y$, a contradiction. 

Let $\tilde{V}_1 := V_1\setminus V_1'$. The following claim is an analog of Claim~\ref{clm:v1} for Lemma \ref{lem3} and can be proved similarly. The only difference between their proofs is that here we find at least $\a n^{k-1}$ copies of $K_{k-1}$ in each of $N(xv)$ and $N(yv)$, which contradicts the definition of $\tilde V_1$.
\begin{claim}\label{clm:2v1}
Given $v\in \tilde V_1$ and $2\le i \le k$, there exists $j\in [k]$ such that $|N_{A_{ij}}(v)|<\a'' t$. \qed
\end{claim}

Fix an vertex $v\in \tilde V_1$. Claim~\ref{clm:2v1} implies that for each $2\le i\le k$, there exists $\ell_i$ such that $|N_{A_{i \ell_i}}(v)| < \a'' t$. Our next claim is an analog of Claim \ref{clm:li=l2} for Lemma \ref{lem3} and can be proved similarly.

\begin{claim}\label{clm:2li=l2}
We have $\ell_2 = \ell_3 = \cdots = \ell_k$. \qed
\end{claim}


We now define $A_{1j}:= \{ v\in \tilde V_1: |N_{A_{2j}}(v)|<\a'' t \}$ for $ j \in [k]$. By Claims~\ref{clm:2v1} and \ref{clm:2li=l2}, this yields a partition of $\tilde V_1 = \bigcup_{j=1}^k A_{1j}$ such that
\begin{equation}\label{eq:2den1}
d(A_{1j}, A_{ij})< \frac{\a''t |A_{1j}|}{|A_{1j}||A_{ij}|} \le \frac{\a''t}{(1 - \a')t} < (1 + 2\a') \a'' \text{ for } i\ge 2 \text{ and } j\ge 1.
\end{equation}
For $v\in A_{1j}$, we have $|N_{A_{ij}}(v)| < \alpha'' t$ for $i\ge 2$. By the minimum degree condition and \eqref{eq:2size},
\begin{equation}
\label{eq:2deg1}
|A^c_{ij} \setminus N(v)| \le (\tfrac1k + \alpha)n - (|A_{ij}|- \alpha''t)< 2\a'' t.
\end{equation}
By \eqref{eq:2size} and \eqref{eq:2deg1}, we derive that
\begin{equation}\label{eq:2den1dia}
\bar d(A_{1j}, A_{ij'})<\frac {|A_{1j}| \cdot 2\a'' t }{|A_{1j}| |A_{ij'}|} \le \frac{2\a'' t}{(1-\a')t} < 3\a'' \text{ for } i\ge 2 \text{ and } j\neq j'.
\end{equation}
We claim that $|A_{1j}| \le (1+ \a)t + (1 + 2\a') \a''|A_{1j}|$ for all $j$. Otherwise, by the minimum degree condition, we have $\deg_{A_{1j}} (v)> (1 + 2\a')\a'' |A_{1j}|$ for all $v\in A_{ij}$, and consequently $d(A_{1j},A_{ij}) > (1 + 2\a') \a''$, contradicting \eqref{eq:2den1}. We thus conclude that
\begin{equation}\label {eq:2size1u}
|A_{1j}|\le \frac{1+\a}{1- (1 + 2\a')\a''} t< (1+2\a'' )\frac nk.
\end{equation}
Since $|V'_1|\le 2\a n$, we have $|\bigcup_{j=1}^{k} A_{1j}| = |V_1\setminus V'_1| \ge |V_1| - 2\a n$.
Using \eqref{eq:2size1u}, we now obtain a lower bound for $|A_{1j}|$, $j\in[k]$.
\begin{equation}\label {eq:2size1l}
|A_{1j}| \ge n - (k-1)(1+ 2 \alpha'' )\frac{n}{k} - 2 \alpha n \ge (1 - 2k\a'')\frac{n}{k}.
\end{equation}

It remains to show that $d(A_{i1},A_{i'1})$ and $d(A_{ik},A_{i'k})$, $2\le i,i'\le k$, are small.
First we show that if both densities are reasonably large then there are too many reachable $(2k-1)$-sets from $x$ to $y$. The proof resembles the one of Claim~\ref{clm:row1}.

\begin{claim} \label{clm:row1k1}
For $2\le i \ne i'\le k$, either $d(A_{i1},A_{i'1})\le 6\a''$ or $d(A_{ik},A_{i'k})\le 6\alpha''$.
\end{claim}

\begin{proof}
Suppose instead, that say $d(A_{(k-1)1},A_{k1}),d(A_{(k-1)k},A_{kk})> 6 \alpha''$. 
Fix a vertex $v_1$ in $A_{1j}$, for some $2\le j\le k-1$, say $v_1\in A_{12}$. 
We construct two vertex disjoint copies of $K_{k-1}$ in $N(xv_1)$ and $N(yv_1)$ as follows. We first select $k-3$ sets $A_{ij}$ with $2\le i\le k-2$ and $3\le j\le k-1$ such that no two of them are on the same row or column -- there are $(k-3)!$ choices. Fix one of them, say
$A_{23},\cdots,A_{(k-2)(k-1)}$. For $2 \le i\le k-2$, we select $v_i\in N_{A_{i(i+1)}}(v_1 \cdots v_{i-1})$ that is $\sqa$-typical to $A_{(k-1)1}$, $A_{k1}$ and $A_{j(j+1)}$, $i< j\le k-2$. By Claim~\ref{clm:2dia} and \eqref{eq:2deg1}, there are at least
\[
(1- (k-2)\sqa) |A_{i(i+1)}| - (k\alpha+\alpha' + \alpha'') \frac{n}{k} \ge \frac{n}{2k}
\]
such $v_i$.
After selecting $v_2, \dots, v_{k-2}$, we select two adjacent vertices $v_{k-1}\in A_{(k-1)1}$ and $v_k\in A_{k1}$ such that $v_{k-1}$ and $v_k$ are in $N(v_1 \cdots v_{k-2})$. For $j=k-1, k$, we know that $N(v_1)$ misses at most $(k\alpha+ \alpha' + \alpha'') n/k$ vertices in $A_{j1}$ and at most $(k-3)\sqa |A_{j1}|$ vertices of $A_{j1}$ are not in $N(v_2 \cdots v_{k-2})$. Since $d(A_{(k-1)1},A_{k1})> 6 \alpha''$, there are at least
\begin{align*}
6\alpha'' |A_{(k-1)1}| |A_{k1}|  - (k\alpha+\alpha' + \alpha'') \frac{n}{k} (|A_{(k-1)1}| + |A_{k1}|) \\
- 2 (k-3)\sqa |A_{(k-1)1}| |A_{k1}| \ge 6 \sqrt{\alpha'} \left(\frac{n}{k} \right)^2
\end{align*}
such pairs $v_{k-1}, v_k$. Hence $N(xv_1)$ contains at least
\[
(k-3)! \left(\frac{n}{2k} \right)^{k-3} 6 \sqrt{\alpha'} \left(\frac{n}{k} \right)^2 \ge \sqrt{\alpha'} \left(\frac{n}{k} \right)^{k-1} \ge \sqrt{\alpha} n^{k-1}
\]
copies of $K_{k-1}$. Let $C$ be such a copy of $K_{k-1}$. Then we follow the same procedure and construct a copy of $K_{k-1}$ on $N(yv_1)\setminus C$. After fixing $k-3$ sets $A_{ij}$ with $2\le i\le k-2$ and $3\le j\le k-1$ such that no two of them are on the same row or column, there are still at least $\frac{n}{2k}$ such $v_i$ for $2\le i\le k-2$.
Then, as $d(A_{ik},A_{i'k})> 6 \alpha''$, there are at least $6 \sqrt{\alpha'} \left(\frac{n}{k} \right)^2 $ choices of $v_{k-1}\in A_{(k-1)k}$ and $v_k\in A_{kk}$ such that $v_{k-1}$ and $v_k$ are in $N(v_1 \cdots v_{k-2})$. This gives at least $\sqrt{\alpha} n^{k-1}$ copies of $K_{k-1}$ in $N(yv_1)$. Then, since there are at least $|V_1| - |A_{11}| - |A_{1k}| \ge \a n$ choices of $v_1$, totally there are at least $\a n (\sqrt{\a} n^{k-1})^2 = \a^2 n^{2k-1}$ reachable $(2k-1)$-sets from $x$ to $y$, a contradiction.
\end{proof}

Next we show that if any of $d(A_{i1},A_{i'1})$ or $d(A_{ik},A_{i'k})$, $2\le i,i'\le k$, is sufficiently large, then we can remove edges from $G$ such that the resulting graph still satisfies the minimum degree condition, which contradicts the assumption that $G$ is minimal.
\begin{claim}\label{clm:row1k2}
For $2\le i \ne i'\le k$, $d(A_{i1},A_{i'1}), d(A_{ik},A_{i'k}) \le 6k \sqrt{\alpha''}$.
\end{claim}

\begin{proof}
Without loss of generality,  assume that $d(A_{2k}, A_{3k}) > 6k \sqrt{\a''}$. By Claim~\ref{clm:row1k1},
we have $d(A_{21},A_{31})< 6\alpha''$. Combining this with \eqref{eq:2den middle}, we have $d(A_{2j},A_{3j})< 6\alpha''$ for all $j\in [k-1]$. Now fix $j\in[k-1]$. The number of non-edges between $A_{2j}$ and $A_{3j}$ satisfies $\bar{e} (A_{2j},A_{3j}) > (1 - 6\a'') |A_{2j}| |A_{3j}|$. By the minimum degree condition and \eqref {eq:2size},
\[
\bar {e} (A_{2k},A_{3j}) < (1 +\ k \alpha) \frac{n}{k} |A_{3j}|-(1- 6\alpha'')|A_{2j}||A_{3j}| \le 7\alpha'' \frac{n}{k} |A_{3j}|. 
\]
Using $\eqref {eq:2size}$ again, we obtain that
\[
d(A_{2k},A_{3j})\ge 1-\frac{7 \a'' \tfrac{n}{k} |A_{3j}|}{|A_{2k}||A_{3j}|}\ge 1-8\alpha''.
\]
This implies that $d(A_{2k},A_{3k}^c)\ge 1- 8\alpha''$. Similarly we derive that $d(A_{3k},A_{2k}^c)\ge 1- 8\alpha''$. For $i = 2,3$, define $A_{ik}^T$ as the set of the vertices in $A_{ik}$ that are  $\sqrt{8\alpha''}$-typical to $A_{(5-i)k}^c$. Thus $|A_{ik}\setminus A_{ik}^T|\le \sqrt{8\alpha''}|A_{ik}|$.

Let $A_{ik}^{T_1}=\{v\in A_{ik}^T: \deg_{A_{(5-i)k}}(v)\le \sqrt{8\alpha''}|A_{jk}^c|\}$ and $A_{ik}^{T_2}=A_{ik}^T\setminus A_{ik}^{T_1}$. For $u\in A_{2k}^{T_2}$, we have
\[
\deg_{V_3} (u) = \deg_{A^c_{3k}}(u) + \deg_{A_{3k}}(u) > ( 1 -\sqrt{8\alpha''} ) |A_{3k}^c| +\sqrt{8\alpha''}|A_{3k}^c| = |A_{3k}^c|.
\]
Since $|A_{3k}^c| \ge \deg_{V_3} (x)$ and $|A_{3k}^c|$ is an integer, we conclude that $\deg_{V_3} (u)\ge  \deg_{V_3} (x) + 1$. Similarly we can derive that $\deg_{V_2} (v)\ge  \deg_{V_2} (x) + 1$ for every $v\in A_{3k}^{T_2}$.
If there is an edge $uv$ joining some $u\in A_{2k}^{T_2}$ and some $v\in A_{3k}^{T_2}$, then we can delete this edge and the resulting graph still satisfies the minimum degree condition. Therefore we may assume that there is no edge between $A_{2k}^{T_2}$ and $A_{3k}^{T_2}$. Then
\begin{align*}
e(A_{2k},A_{3k})& \le e(A_{2k}\setminus A_{2k}^T,A_{3k}) + e(A_{2k},A_{3k}\setminus A_{3k}^T) + e(A_{2k}^{T_1}, A_{3k}^T) + e(A_{2k}^{T}, A_{3k}^{T_1}) \\
                            &\le 2\sqrt{ 8\alpha''}|A_{2k}||A_{3k}|+|A_{2k}^{T_1}|\sqrt{8\alpha''}|A_{3k}^c|+|A_{3k}^{T_1}|\sqrt{8\alpha''}|A_{2k}^c|\\
                            &\le \sqrt{8\alpha''} \left( 2|A_{2k}||A_{3k}| + |A_{2k}| |A_{3k}^c| + |A_{3k}| |A_{2k}^c| \right) \\
                            &= \sqrt{8\alpha''} \left(|A_{2k}| |V_3|+ |A_{3k}| |V_2| \right)\\
                            &\le 3\sqrt{\alpha''} \cdot 2k |A_{2k}||A_{3k}| \quad \text{by } \eqref{eq:2size}.
\end{align*}
Therefore $d(A_{2k},A_{3k})\le 6k \sqrt{\alpha''}$.
\end{proof}

In summary, by \eqref{eq:2size}, \eqref{eq:2size1u} and \eqref{eq:2size1l}, we have $(1-  2k\alpha'') \frac{n}{k} \le |A_{ij}| \le (1 + 2\alpha'' )\frac{n}{k} $ for all $i$ and $j$. In order to make $\bigcup_{j=1}^{k} A_{ij}$ a partition of $V_i$, we move the vertices of $V_1'$ to $A_{11}$ and the vertices of $A_{i0}$ to $A_{i2}$ for $2\le i\le k$. Since $|V_1'| < 2\a n$ and $|A_{i0}| \le (k+1)\a n$, we have $| |A_{ij}| - \tfrac nk | \le 2k \alpha'' \frac{n}{k}$ after moving these vertices. On the other hand, by \eqref{eq:2den middle}, \eqref{eq:2den1}, and Claim~\ref{clm:row1k2}, we have $d(A_{ij}, A_{i'j}) \le 6k \sqrt{\alpha''}$ for $i\ne i'$ and all $j$. (In fact, for $i\ge 2$, we now have $d(A_{11}, A_{i1}) \le 2\alpha''$ as we added at most $2\a n$ vertices to $A_{11}$. For $i'> i\ge 2$, we now have $d(A_{12}, A_{i2}) \le 2\alpha''$ and $d(A_{i2}, A_{i'2}) \le 2\alpha'$ as we moved at most $(k+1)\alpha n$ vertices to $A_{i2}$.) Therefore after deleting edges, $G$ is ($ 2k \alpha'', 6k \sqrt{\alpha''}$)-approximate to $\Theta_{k\times k} (n/k)$. By \eqref{eq:ad}, and the definitions of $\alpha''$ and $\alpha'$,  $G$ is $(\Delta/6, \Delta/2)$-approximate to $\Theta_{k\times k} (n/k)$.
\end{proof}

\section*{Acknowledgments}
The authors thank two referees for their valuable comments that improved the presentation and shortened the proofs of Lemmas \ref{lem:reach} and \ref{lem3}.

\end{document}